\documentclass[12pt,psamsfonts]{amsart}
\usepackage{fullpage,amssymb,amsfonts,amsmath}
\usepackage[all,cmtip]{xy}
\usepackage{enumerate}
\usepackage{mathrsfs}
\usepackage{comment}
\usepackage{hyperref}
\usepackage[capitalise]{cleveref}
\usepackage{floatrow}
\usepackage{url}
\let\oldtocsection=\tocsection

\let\oldtocsubsection=\tocsubsection

\let\oldtocsubsubsection=\tocsubsubsection

\renewcommand{\tocsection}[2]{\hspace{0em}\oldtocsection{#1}{#2}}
\renewcommand{\tocsubsection}[2]{\hspace{1em}\oldtocsubsection{#1}{#2}}
\renewcommand{\tocsubsubsection}[2]{\hspace{2em}\oldtocsubsubsection{#1}{#2}}
\theoremstyle{plain} 

\newtheorem{thm}{Theorem}[section]
\newtheorem{cor}[thm]{Corollary}
\newtheorem{prop}[thm]{Proposition}
\newtheorem{lem}[thm]{Lemma}

\theoremstyle{definition}

\theoremstyle{remark}

\newtheorem*{rems*}{Remarks}

\numberwithin{equation}{section}

\DeclareUrlCommand\DOI{}
\newcommand{\doi}[1]{\href{http://dx.doi.org/#1}{\DOI{doi:#1}}}



\crefname{figure}{Figure}{Figures}
\theoremstyle{plain}
\newtheorem*{thm*}{Theorem}
\crefname{thm}{Theorem}{Theorems}
\crefname{cor}{Corollary}{Corollarys}
\newtheorem*{cor*}{Corollary}
\crefname{cor*}{Corollary}{Corollarys}
\crefname{lem}{Lemma}{Lemmas}
\crefname{prop}{Proposition}{Propositions}
\crefname{conj}{Conjecture}{Conjectures}
\newtheorem*{conj*}{Conjecture}
\crefname{conj*}{Conjecture}{Conjectures}
\crefname{defn}{Definition}{Definitions}

\theoremstyle{remark}
\newtheorem*{rem*}{Remark}

\def\addsymbol #1: #2#3{$#1$ \> \parbox{5.4in}{#2 \dotfill \pageref{#3}}\\} 
\def\addsymbolEND #1: #2#3{$#1$ \> \parbox{5.4in}{#2 \dotfill \pageref{#3}}}

\newcommand{\ds}{\displaystyle}
\newcommand{\fk}[1]{\mathfrak{#1}}

\newcommand{\cA}{\mathcal{A}}

\renewcommand{\bar}{\overline}

\newcommand{\C}{\mathbb{C}}
\newcommand{\cC}{\mathcal{C}}
\newcommand{\Cl}{\mathrm{Cl}}

\newcommand{\kd}{\mathfrak{d}}
\newcommand{\cD}{\mathcal{D}}

\newcommand{\kf}{\mathfrak{f}}

\newcommand{\km}{\mathfrak{m}}

\renewcommand{\pmod}[1]{\, (\mathrm{mod} {\, #1})}
\newcommand{\kn}{\mathfrak{n}}
\newcommand{\N}{\mathbb{N}}

\newcommand{\kp}{\mathfrak{p}}
\newcommand{\cP}{\mathcal{P}}

\newcommand{\kP}{\mathfrak{P}}

\newcommand{\cO}{\mathcal{O}}

\newcommand{\kq}{\mathfrak{q}}
\newcommand{\Q}{\mathbb{Q}}

\newcommand{\R}{\mathbb{R}}
\renewcommand{\Re}{\mathrm{Re}}
\def\Res{\mathop{\mathrm{Res}}}

\newcommand{\Z}{\mathbb{Z}}

\title{On the Least Prime Ideal and Siegel Zeros}

\author{Asif Zaman}
\thanks{The author was supported in part by an NSERC PGS-D scholarship.} 
\address{
Department of Mathematics, University of Toronto \\
Room 6290, 40 St. George St., M5S2E4, Toronto, ON, Canada}
\email{asif@math.toronto.edu}

\date{\today}

\begin{document}

\begin{abstract} Let $K$ be a number field, $\kq$ be an integral ideal, and $\Cl(\kq)$ be the associated ray class group. Suppose $\Cl(\kq)$ possesses a real exceptional character $\psi$, possibly principal, with real zero $\beta$. If $\beta$ is a Siegel zero and $\cC \in \Cl(\kq)$ satisfies $\psi(\cC) = 1$ then we show there exists a prime ideal $\kp \in \cC$ such that
\[
\N\kp \ll_{\delta} \big\{ n_K^{16n_K} \cdot d_K^{9.5} \cdot (\N\kq)^{9} \big\}^{1+\delta} e^{O_{\delta}(n_K)}
\]
where $n_K$ is the degree of $K/\Q$, $d_K = |\mathrm{disc}(K/\Q)|$ is the absolute discriminant of $K$, and $\N = \N^K_{\Q}$ is the absolute norm of $K$. All implicit constants are effective. 

A special case of this result is related to rational primes represented by certain binary quadratic forms. 
\end{abstract}

\maketitle

\tableofcontents

\section{Introduction}

Let $K$ be a number field, $\cO$ its ring of integers, and $\kq \subseteq \cO$ be an integral ideal. The (narrow) ray class group of $K$ modulo $\kq$, denoted $\Cl(\kq)$, is the quotient of fractional ideals of $K$ relatively prime to $\kq$ and principal ideals $(\alpha)$ such that $\alpha \equiv 1 \pmod{\kq}$ and $\alpha$ is totally positive.  For a given class $\cC \in \Cl(\kq)$, it has long been known that there are infinitely many prime ideals $\kp \in \cC$. Therefore it is natural to ask: 
\begin{center}
\emph{What is the least norm of a prime ideal $\kp \in \cC$?}
\end{center}
We refer to this question as the ``least prime ideal" problem. The Generalized Riemann Hypothesis (GRH) for Hecke $L$-functions implies for $\delta > 0$,
\begin{equation}
\N\kp \ll_{\delta} (d_K \N\kq)^{\delta} \cdot h(\kq)^{2+\delta}
\label{GRHBound}
\end{equation}
where $d_K = |\mathrm{disc}(K/\Q)|$ is the absolute discriminant of $K$, $\N = \N^K_{\Q}$ is the absolute norm of $K$, and $h(\kq) = \#\Cl(\kq)$ is the size of the ray class group. Fogels \cite{Fog4} was the first to give an unconditional answer showing 
\[
\N\kp \ll_K (\N\kq)^{C_K}
\]
but this bound is not entirely satisfactory because  the implied constant and exponent depend on $K$ in an unspecified manner. In his Ph.D. thesis work, Weiss \cite{Weiss} proved a $K$-uniform version of Fogels' result; that is, unconditionally
\begin{equation}
\N\kp \ll n_K^{An_K} \cdot d_K^B \cdot (\N\kq)^C 
\label{WeissBound}
\end{equation}
where $n_K = [K:\Q]$ is the degree of $K$ and $A,B,C > 0$ are absolute constants. Assuming GRH, one may take $(A,B,C) = (\delta,1+\delta,2+\delta)$ for $\delta > 0$. The focus of this paper is, in an exceptional case, to exhibit a bound like \eqref{WeissBound} with explicit exponents. 

Specializing to $K = \Q$ and $\kq = (q)$, the least prime ideal problem naturally corresponds to the least prime $p$ in an arithmetic progression $a \pmod{q}$. Linnik  \cite{Linnik1} famously showed  unconditionally that 
\begin{equation*}
p \ll q^L
\label{Linnik_Constant}
\end{equation*}
for some absolute constant $L > 0$ known as ``Linnik's constant" and where the implicit constant is effective. Conjecturally, $L=1+\delta$ for any $\delta > 0$ is admissible and GRH implies $L=2+\delta$ is acceptable. Since Linnik, many authors have computed admissible values of $L$ (see the landmark paper of Heath-Brown \cite{HBLinnik} for details) with the current world record being $L=5.2$ by Xylouris \cite{Xylouris}. 

Thus far, a crucial ingredient to all proofs computing Linnik's constant is the handling of a putative real zero 
\[
\beta = 1 - \frac{1}{\eta \log q}
\] 
of a Dirichlet $L$-function attached to a quadratic Dirichlet character $\psi \pmod{q}$. If $\eta \geq 3$ we refer to this scenario as the \emph{exceptional case} and the zero $\beta$ as an \emph{exceptional zero}. If additionally $1/\eta = o(1)$, then we call $\beta$ a \emph{Siegel zero} which conjecturally does not exist. Most authors adapted Linnik's original proof and established a quantitative Deuring-Heilbronn phenomenon which is a strong form of zero repulsion for $\beta$. However, in the exceptional case, the best bound thus far on Linnik's constant involves sieve methods and was pioneered by Heath-Brown \cite{HBSiegel}. He showed, with effective implicit constants, that $L=3+\delta$ is an admissible value provided $\eta \ge \eta(\delta)$ which bests the aforementioned unconditional $L=5.2$. Even more astonishingly, Heath-Brown showed that the GRH bound $L=2+\delta$ is an admissible value provided $\eta \geq \eta(\delta)$ although the implied constants are ineffective. Sieve techniques are indeed very advantageous in the exceptional case. To further emphasize this point, we remark that Friedlander and Iwaniec \cite{FI_Siegel-LPAP} proved, under some additional technical assumptions, that $L=2-\tfrac{1}{59}$ is admissible when a Siegel zero exists. This surpasses GRH! 

Now, let us describe the exceptional case in the context of the least prime ideal problem for a number field $K$. A character $\chi \in \Cl(\kq)$ is known as a Hecke character, denoted $\chi \pmod{\kq}$. We may pullback its domain and extend it by zero to all integral ideals of $K$; that is, $\chi(\kn) = 0$ if $(\kn,\kq) \neq (1)$. Then its  associated Hecke $L$-function is defined to be:
\[
L(s,\chi) = \sum_{ \kn \subseteq \cO} \chi(\kn) (\N{\kn})^{-s} = \prod_{\kp} \Big(1-\frac{\chi(\kp)}{(\N{\kp})^s} \Big)^{-1} 
\]
for $\sigma > 1$ where $s = \sigma+it$. These are the usual Dirichlet $L$-functions modulo $q$ when $K=\Q$ and $\kq = (q)$. Naturally, the zeros of Hecke $L$-functions are intimately related to the distribution of prime ideals $\kp$ of $K$ within classes of $\Cl(\kq)$. It is well-known that Hecke $L$-functions admit a meromorphic continuation to $\C$ with only one simple pole at $s=1$ if $\chi$ is the principal character. Further, they possess a (nearly) zero-free region of the form
\[
\sigma \geq 1 - \frac{c}{\log( n_K^{n_K} d_K \N\kq)}, \qquad |t| \leq 1,
\]
where $c > 0$ is an absolute constant.  However, just as with Dirichlet $L$-functions, exactly one real zero $\beta$ attached to a real character $\psi \pmod{\kq}$ cannot be eliminated from this region -- no matter how small $c$ is chosen. See for example \cite[Lemma 2.3]{LMO} or, for an explicit version, \cite{Zaman2015} where $c=0.0875$ is shown to be admissible. We emphasize that $\psi$ may be quadratic or principal.

For the remainder of the paper, suppose $\psi \pmod{\kq}$ is a real Hecke character with a real zero
\begin{equation}
\beta = 1 - \frac{1}{\eta \log( n_K^{n_K} d_K \N\kq)}
\label{RealZero}
\end{equation}
where $\eta \geq 20$; that is, $\beta$ is an \emph{exceptional zero} of the \emph{exceptional character} $\psi$.  If $1/\eta = o(1)$ then we shall call  $\beta$ a \emph{Siegel zero}.  For a ray class $\cC \in \Cl(\kq)$ satisfying $\psi(\cC) = 1$, we establish an explicit effective $K$-uniform bound for the size of the least prime ideal $\kp \in \cC$ provided $\beta$ is a Siegel zero. 

\begin{thm} \label{Theorem1-LPI} Let $K$ be a number field and $\kq$ an integral ideal. Suppose $\psi \pmod{\kq}$ is a real Hecke character such that $L(s,\psi)$ has a real zero $\beta$ as in \eqref{RealZero}.  Let $\cC \in \Cl(\kq)$  satisfy $\psi(\cC) = 1$ and $\delta > 0$ be given.  Then there exists a prime ideal $\kp \in \cC$ satisfying
\[
\N\kp \ll_{\delta}  \big\{ n_K^{An_K} \cdot d_K^{B} \cdot (\N\kq)^{C} \cdot h(\kq)^2 \big\}^{1+\delta} e^{O_{\delta}(n_K)} 
\]
provided $\eta \geq \eta(\delta)$ and where
\begin{equation}
(A,B,C) = 
\begin{cases}
(16, 6+\tfrac{5}{n_K}, 5 + \tfrac{2}{n_K}) & \text{if $\psi$ is quadratic},  \\
(6, 3 + \tfrac{4}{n_K}, 3) & \text{if $\psi$ is principal}. 
\end{cases}
\label{Theorem1-LPI_exponents}
\end{equation}
All implicit constants are effective. 
\end{thm}
\begin{rems*}  $ $
\begin{enumerate}
\item 
\label{Theorem1-LPI_R_Simple} The factor of $h(\kq)^2$ is natural in light of \eqref{GRHBound} but one may prefer a bound similar to \eqref{WeissBound}. Using \cref{PR-RayClassGroup} allows us to give the alternative bound
\[
\N\kp \ll_{\delta} \big\{ n_K^{A'n_K} \cdot d_K^{B'} \cdot (\N\kq)^{C'} \big\}^{1+\delta} e^{O_{\delta}(n_K)}
\]
with 
	\[
	(A',B',C') = 
\begin{cases}
(16, 7+\tfrac{5}{n_K}, 7 + \tfrac{2}{n_K}) & \text{if $\psi$ is quadratic},  \\
(6, 4 + \tfrac{4}{n_K}, 5) & \text{if $\psi$ is principal}. 
\end{cases}
	\]
	Even more simply, $(A',B',C') = (16, 9.5, 9)$ is admissible in all cases. 

\item \label{Theorem-LPI_R-Q} For a point of reference, consider the estimate in the special case $K=\Q$ and $\kq = (q)$. If there exists a quadratic Dirichlet character $\psi \pmod{q}$ with real zero $\beta = 1 - \frac{1}{\eta \log q}$ and $\psi(a) = 1$ for $(a,q) = 1$, then \cref{Theorem1-LPI} implies there exists a prime $p \equiv a \pmod{q}$  such that 
	\[
	p \ll_{\delta} q^{9+\delta}
	\]
	provided $\eta \geq \eta(\delta)$. The exponent $L = 9+\delta$ is comparable to the unconditional $L=5.2$ by Xylouris  \cite{Xylouris} and to the effective Siegel zero case $L=3+\delta$ by Heath-Brown \cite{HBSiegel}.
	
\item  
\label{Theorem1-LPI_R_Ineffective}
By a straightforward modification, one can improve  \cref{Theorem1-LPI} by appealing to the Brauer-Siegel Theorem (see \cref{kappa_LB_ineffective}) from which it follows
	\[
	(A,B,C) = 
\begin{cases}
(6, 6, 5) & \text{if $\psi$ is quadratic}, \\
(2, 3, 3) & \text{if $\psi$ is principal}, 
\end{cases}
	\]
	or as in Remark 1,
	\[
	(A',B',C') = 
\begin{cases}
(6, 7, 7) & \text{if $\psi$ is quadratic}, \\
(2, 4, 5) & \text{if $\psi$ is principal}, 
\end{cases}
	\]
	but the implicit constants are \emph{ineffective}. 
\end{enumerate}
\end{rems*}
If $K$ is an imaginary quadratic field of discriminant $D$ then the ray class group $\Cl(\cO)$ has a well-known\footnote{See \cite[Theorem 7.7]{Cox} for example.} correspondence with the group of  form classes of primitive positive-definite integral  binary quadratic forms of discriminant $D$. Under this bijection, such a  form $Q(x,y)$ represents an integer $m$ if and only if its corresponding ray class $\cC \in \Cl(\cO)$ contains an integral ideal $\km$ satisfying $\N\km = m$. With this interpretation, \cref{Theorem1-LPI} has an analogous result in this special case. 

\begin{cor} \label{Theorem1-LPI_QuadraticForm} Let $K$ be the imaginary quadratic field of discriminant $D \geq 1$.  Suppose $\psi \pmod{\cO}$ is a real Hecke character such that $L(s,\psi)$ has a real zero $\beta$ as in \eqref{RealZero}.  Let $\delta > 0$ and $Q(x,y)$ be a primitive integral positive-definite binary quadratic form of discriminant $D$ in natural correspondence with a ray class $\cC \in \Cl(\cO)$ satisfying $\psi(\cC) = 1$. 

Then there exists a rational prime $p$ such that $p = Q(x,y)$ has a solution $(x,y) \in \Z^2$ and 
\[
p \ll_{\delta} 
\begin{cases}
D^{9.5 + \delta} 
& \text{if $\psi$ is quadratic}, \\
D^{6+\delta}   & \text{if $\psi$ is principal}, 
\end{cases}
\]
provided $\eta \geq \eta(\delta)$. All implicit constants are effective. 
\end{cor}
\begin{rems*} $ $
\begin{enumerate}
\item As per Remark 3 following \cref{Theorem1-LPI}, one can sharpen the bound in \cref{Theorem1-LPI_QuadraticForm} to
\[
p \ll_{\delta} 
\begin{cases}
D^{7 + \delta} 
& \text{if $\psi$ is quadratic}, \\
D^{4+\delta}   & \text{if $\psi$ is principal}, 
\end{cases}
\]
but the implicit constants are rendered ineffective. 
\item For frame of reference, one may indirectly compare \cref{Theorem1-LPI_QuadraticForm} with an unconditional bound for $p$ on average due to Ditchen \cite{Ditchen}.  Informally speaking, he showed forms of discriminant $D \equiv 0 \pmod{8}$  represent some prime $p$ satisfying
\[
p \ll_{\delta} \begin{cases}  D^{20/3 + \delta} & \text{on average over discriminants $D$},  \\ 
D^{3+\delta} & \text{on average over discriminants $D$ and form classes}.
\end{cases}
\]
As far as the author is aware, \cref{Theorem1-LPI_QuadraticForm} is the first result to bound the least prime represented by quadratic forms with an explicit exponent  uniformly over all discriminants, albeit conditionally in an exceptional case. 
\end{enumerate}
\end{rems*}
\cref{Theorem1-LPI,Theorem1-LPI_QuadraticForm} are both straightforward consequences of the following quantitative lower bound for the number of prime ideals in a given ray class. Here $\kappa_K$ is the residue at $s=1$ of the Dedekind zeta function $\zeta_K(s)$ and 
\[
\varphi_K(\kq) = \N\kq \prod_{\kp \mid \kq} \Big(1- \frac{1}{\N\kp}\Big)
\]
is the generalized Euler $\varphi$-function of $K$. 
\begin{thm} \label{Theorem1} Let $K$ be a number field and $\kq$ an integral ideal. Suppose $\psi \pmod{\kq}$ is a real Hecke character such that $L(s,\psi)$ has a real zero $\beta$ as in \eqref{RealZero}.  Let $\cC \in \Cl(\kq)$  satisfy $\psi(\cC) = 1$.  For $\delta > 0$, assume $\eta \geq \eta(\delta)$ and $M_{\delta} > 0$ are sufficiently large.  Further assume
\begin{equation}
e^{M_{\delta}n_K} \big\{ n_K^{A n_K} \cdot d_K^{B} \cdot (\N\kq)^{C} \cdot h(\kq)^2  \big\}^{1+\delta} \leq x  \leq  e^{M_{\delta}n_K}(n_K^{n_K} d_K \N\kq)^{100},
\label{Theorem1-xRange}
\end{equation}
where $(A,B,C)$ are given by \eqref{Theorem1-LPI_exponents}. Then
\begin{equation}
\label{Theorem1_LB}
 \#\{ \kp \in \cC \text{ prime} : \N\kp <x \}  \geq  c_{\psi} \Delta_{\psi} \cdot \kappa_K \frac{\varphi_K(\kq) }{\N\kq}   \cdot \frac{x}{h(\kq)}
\end{equation}
where
\begin{equation*}
\Delta_{\psi} =  
\begin{cases}
\ds L(1,\psi)  \prod_{\psi(\kp) = 1} \Big(1-\frac{3}{\N\kp^2} + \frac{2}{\N\kp^3} \Big) \prod_{\psi(\kp)=-1} \Big( 1 - \frac{1}{\N\kp^2} \Big) & \text{if $\psi$ is quadratic}, \\
\ds \prod_{\kp \nmid \kq} \Big( 1 - \frac{1}{\N\kp^2} \Big) & \text{if $\psi$ is principal},
\end{cases}
\end{equation*}
and
\begin{equation*}
c_{\psi} =
\begin{cases}
0.00466 & \text{if $\psi$ is quadratic}, \\
 0.0557 & \text{if $\psi$ is principal}. 
\end{cases}
\end{equation*}
All implicit constants are effectively computable. 

\end{thm}
\begin{rems*} $ $
\begin{enumerate}
	\item Bounding $h(\kq)$ by \cref{PR-RayClassGroup}, we see that \eqref{Theorem1-xRange} contains the interval
	\[
	e^{M_{\delta}'n_K} \big\{ n_K^{A' n_K} \cdot d_K^{B'} \cdot (\N\kq)^{C'}  \big\}^{1+\delta} \leq x  \leq  e^{M_{\delta}n_K}(n_K^{n_K} d_K \N\kq)^{100}. 
	\]
	where $(A',B',C')$ are given by Remark 1 following \cref{Theorem1-LPI} and $M_{\delta}' = M_{\delta} + 2+2\delta$. 
	\item According to Remark 3 following \cref{Theorem1-LPI}, one can widen the lower bound of interval \eqref{Theorem1-xRange} using the ineffective Brauer-Siegel Theorem. 
	\item By obvious modifications to the proof, one can easily obtain an upper bound of the same form as \eqref{Theorem1_LB}.   That is, for the same range as \eqref{Theorem1-xRange}, one can show 
	\[
	 \#\{ \kp \in \cC \text{ prime} : \N\kp <x \} \leq \tilde{c}_{\psi}  
	 \Delta_{\psi} \cdot \kappa_K \frac{\varphi_K(\kq) }{\N\kq}   \cdot \frac{x}{h(\kq)}
	\]
	where
	\[
	\tilde{c}_{\psi} = 
	\begin{cases}
8.62 & \text{if $\psi$ is quadratic}, \\
4.02 & \text{if $\psi$ is principal}. 
\end{cases}
	\]
	 Upper bounds for even wider ranges of $x$ could potentially also be established by allowing for a constant larger than $\tilde{c}_{\psi}$. 
	\item The constant $c_{\psi}$ is likely subject to improvement which we do not seriously pursue here as that is not our aim. 
	
	\item  One can also establish a variant of \cref{Theorem1} which holds for larger values of $x$.  For instance, one could instead assume
		\[
		(e^{M_{\delta} n_K} \cdot n_K^{n_K} d_K \N\kq)^{\ell}  \leq x \leq (e^{M_{\delta} n_K} \cdot n_K^{n_K} d_K \N\kq)^{100\ell}
		\]
		for any integer $\ell \geq 20$, say. Adapting the argument in \cref{subsec:T1-Proof}, one can deduce the same lower bound with 
		\[
	c_{\psi} = 
	\begin{cases}
0.0275 - O(\frac{e^{\ell}}{\ell!}) & \text{if $\psi$ is quadratic}, \\
0.0749- O(\frac{e^{\ell}}{\ell!})  & \text{if $\psi$ is principal},
\end{cases}
	\]
	 and provided $\eta \geq \eta(\delta, \ell)$. 
		
		
\end{enumerate}
\end{rems*}
The primary objective of this paper is to prove \cref{Theorem1}.  The arguments involved are motivated by  the sieve-based techniques employed for the classical case $K= \Q$, including Heath-Brown's aforementioned foundational paper \cite{HBSiegel} and an elegant modern proof by Friedlander and Iwaniec \cite[Chapter 24]{Opera}. To be more specific, let us sketch the main components, and for concreteness temporarily suppose that $\psi \pmod{\kq}$ is quadratic. First, we establish the Fundamental Lemma (\cref{FundamentalLemma})  for zero-dimensional sieves in number fields and aim to apply it a sequence $\{ a_{\kn}\}_{\kn \subseteq \cO}$ where
\[
a_{\kn} \approx \mu_K^2(\kn) \mathbf{1}\{\kn \in \cC\} \cdot \sum_{\kd \mid \kn} \psi(\kd),
\]
$\mu_K(\, \cdot \,)$ is the M\"{o}bius function defined by \eqref{KMobius}, and $\mathbf{1}\{\, \cdot \,\}$ is an indicator function. Roughly speaking, the sum $\sum_{\kd \mid \kn} \psi(\kd)$ pretends to be an indicator function for integral ideals $\kn$ satisfying $\kp \mid \kn \implies \psi(\kp) = 1$.  After computing local densities, we show that our sieve problem is zero-dimensional because $\psi(\cC) = 1$ and a Siegel zero is assumed to exist. Then we use a Buchstab identity and apply the Fundamental Lemma to lower bound terms with no small prime ideal factors and upper bound terms with large prime ideal factors. An appropriate choice of the relevant sieve parameters and a Tauberian-type argument finishes the proof.

Proving a version of \cref{Theorem1} for the non-residue case $\psi(\cC) = -1$ would certainly be desirable but it is not immediately clear how to do so by sieve-based techniques. In the classical case $K=\Q$, the corresponding sieve problem is one-dimensional leading to an excellent value for Linnik's constant which was first established by Heath-Brown \cite{HBSiegel}. For a general number field $K$ of degree $n_K$, if most small rational primes split then the sieve problem could at worst have dimension $n_K$. Since we seek a bound like  \eqref{WeissBound} with absolute exponents, this high dimension issue therefore poses a difficulty when $\psi(\cC) = -1$. 

Finally, we summarize the organization of this paper. \cref{sec:Sieve_NT} sets up a sieve in number fields and proves the Fundamental Lemma for zero-dimensional sieves. The discussion therein is a close adaptation of \cite[Chapters 5 \& 6]{Opera} but is included for completeness as many variations of number field sieves exist. \cref{sec:Prelim} consists of background material on Hecke $L$-functions, elementary estimates, and notation which will be used throughout the paper. 
\cref{sec:SieveApplication} computes the key components of our sieve problem -- local densities and dimension -- and estimates terms with small prime factors and large prime factors. \cref{sec:TheoremProof} contains the proof of \cref{Theorem1}. 

\subsection*{Acknowledgements} I am very happy to acknowledge the patience, support,  and encouragement from my advisor, Prof. John Friedlander. Always generous with his time and advice, he initially suggested this intriguing problem to me and has provided many helpful comments during our discussions.

\section{Sieve theory in number fields}
\label{sec:Sieve_NT}
\subsection{Notation}
Begin with a sequence $\cA = \{ a_{\kn}\}_{\kn \subseteq \cO}$ of non-negative real numbers such that
\[
|\cA| := \sum_{\kn \subseteq \cO} a_{\kn}
\]
converges\footnote{For instance, one could take $a_{\kn} = e^{-\N\kn/x}$ with $x \geq 1$.}. For an integral ideal $\kd \subseteq \cO$, define
\[
\cA_{\kd} = \{ a_{\kn} : \kd \mid \kn\} \qquad |\cA_{\kd}| := \sum_{\kd \mid \kn} a_{\kn} 
\]
and suppose 
\[
|\cA_{\kd}| = g(\kd) X + r_{\kd}
\]
for some multiplicative function $g(\kd)$ called the \emph{density function} and \emph{remainders} $r_{\kd}$. The \emph{local densities} $g(\kd)$ satisfy
\[
0 \leq g(\kp) < 1
\]
for all prime ideals $\kp$ of $\cO$. Given a set of prime ideals $\cP$ and \emph{sifting level} $z \geq 2$, define
\[
\kP = \kP(z) := \prod_{\substack{ \kp \in \cP \\ \N\kp < z} } \kp, \qquad V(z) := \prod_{\substack{ \kp \in \cP \\ \N\kp < z} } (1-g(\kp)),
\]
and
\[
S(\cA, \cP, z) = S(\cA, z) :=  \sum_{(\kn, \kP(z)) = 1} a_{\kn}
\]
where we suppress the dependence on $\cP$ or $z$ when it is understood. Recall the M\"{o}bius function $\mu_K(\, \cdot \,)$ on integral ideals is defined by
\begin{equation}
\mu_K(\kn) = \begin{cases}
(-1)^r & \text{if $\kn = \kp_1 \cdots \kp_r$ where $\kp_i$ are distinct prime ideals,} \\
0 & \text{otherwise},
\end{cases}
\label{KMobius}
\end{equation}
or equivalently
\begin{equation}
\label{KMobius-2}
 \sum_{\kd \mid \kn} \mu_K(\kd) = \begin{cases} 1 & \text{if $\kn = (1)$,} \\ 0 & \text{otherwise.} \end{cases}
\end{equation}

Sifting $\cA$ according to $\cP$ amounts to estimating $S(\cA, z)$. It is therefore natural to introduce a function, called the \emph{sieve weight},
\[
\Lambda = (\lambda_{\kd})_{\kd}, \qquad \text{for } \kd \mid \kP(z) \text{ and } \N\kd < D
\]
which acts as a finite approximation to the M\"{o}bius function with \emph{level of distribution} $D$. From \eqref{KMobius-2}, one can easily see that
\[
S(\cA, z) = \sum_{\kd \mid \kP(z)} \mu(\kd) |\cA_{\kd}| 
\]
so our approximation takes the form
\[
S^{\Lambda}(\cA,z) :=  \sum_{\kd} \lambda_{\kd} |\cA_{\kd}| = \sum_{\kn} a_{\kn} \Big( \sum_{\kd \mid \kn} \lambda_{\kd} \Big).  
\]
Of special importance are weights $\Lambda^+ = (\lambda^+_{\kd})$ and $\Lambda^- = (\lambda^-_{\kd})$ satisfying
\begin{equation}
\sum_{\kd \mid \kn} \lambda_{\kd}^- \leq \sum_{\kd \mid \kn} \mu_K(\kd) \leq \sum_{\kd \mid \kn} \lambda^+_{\kd}
\label{UpperAndLowerBoundSieveWeights}
\end{equation}
and therefore implying
\begin{equation}
S^-(\cA, z) \leq S(\cA, z) \leq S^+(\cA,z)
\label{UpperAndLowerBoundSieve}
\end{equation}
where the \emph{lower bound sieve}  $S^-$ and the \emph{upper bound sieve} $S^+$ correspond to $\Lambda^-$ and $\Lambda^+$ respectively.  In keeping with notation, we naturally define the main term sums by
\[
V^+(D,z) = \sum_{\substack{ \kd \mid \kP(z) \\ \N\kd < D}} \lambda_{\kd}^+ g(\kd), \qquad V^-(D,z) = \sum_{\substack{ \kd \mid \kP(z) \\ \N\kd < D}} \lambda_{\kd}^- g(\kd),
\]
and remainder terms by
\[
R^+(D,z) = \sum_{\substack{ \kd \mid \kP(z) \\ \N\kd < D}} \lambda_{\kd}^+ r_{\kd}, \qquad R^-(D,z) = \sum_{\substack{ \kd \mid \kP(z) \\ \N\kd < D}} \lambda_{\kd}^- r_{\kd}.
\]
The conditions under the sums may be dropped in light of the definition of the sieve weights, but we include them for emphasis and clarity.

We will be concerned with sieves satisfying
\begin{equation}
\frac{V(w)}{V(z)} = \prod_{w \leq \N\kp < z} \big(1-g(\kp))^{-1} \leq C \Big( \frac{\log z}{\log w} \Big)^{\kappa} \qquad \text{for $2 \leq w < z$,}
\label{SieveDimension}
\end{equation}
where $C > 1$ is a constant  and $\kappa \geq 0$ is the \emph{sieve dimension}. 

\subsection{Buchstab Iterations}

Fix a norm-based total ordering ``$\prec$" of prime ideals of $\cO$; that is, for prime ideals $\kp$ and $\kp'$,
\[
\kp \prec \kp' \implies \N\kp \leq \N\kp'. 
\] 
Abusing notation, for $y \in \R$, write $y \prec \kp$ (resp. $\kp \prec y$) if $y < \N\kp$ (resp. $\N\kp < y$). Observe
\begin{equation}
\N\kp \preceq \kp \text{ and } \kp \preceq \N\kp, \quad \text{ but } \quad  \N\kp \not\prec \kp \text{ and }  \kp \not\prec \N\kp
\label{IdealNormOrdering}
\end{equation}
with this choice. Further abusing notation, for a prime ideal $\km$, define
\[
\kP(\km) :=  \prod_{\substack{\kp \in \cP \\ \kp \prec \km}}  \kp, \qquad V(\km) := \prod_{\substack{\kp \in \cP \\ \kp \prec \km}} (1-g(\kp)), 
\]
and
\[
S(\cA, \km) := \sum_{(\kn, \kP(\km) )=1} a_{\kn}.
\]
Comparing with notation from the previous subsection and using \eqref{IdealNormOrdering}, notice
\[
\kP(\N\km) \mid  \kP(\km), \qquad V(\km) \leq V(\N\km), \qquad \text{ and } \qquad S(\cA, \km) \leq S(\cA, \N\km). 
\]
Note that the results of this paper are independent of the choice of ordering. 

Now, choose sieve weights $\Lambda^+ = (\lambda^+_{\kd})$ and $\Lambda^- = (\lambda^-_{\kd})$ defined to be the M\"{o}bius function truncated to sets of the type
\begin{equation}
\begin{aligned}
\cD^+ & := \{ \kd = \kp_1\cdots \kp_{\ell} : \kp_m \prec y_m \quad \text{for $m$ odd} \} \\
\cD^- & := \{ \kd = \kp_1\cdots \kp_{\ell} : \kp_m \prec y_m \quad \text{for $m$ even} \} \\
\end{aligned}
\label{TruncationSets}
\end{equation}
where $\kd$ is written as a product of distinct prime ideals enumerated in decreasing order,
\[
\kd = \kp_1\cdots\kp_{\ell} \quad \text{with $z \succ \kp_1 \succ \cdots \succ \kp_{\ell}$}.
\]
By convention, $\cD^+$ and $\cD^-$ both contain $\kd = (1)$. The real numbers $y_m$ are \emph{truncation parameters} and by inclusion-exclusion, \eqref{UpperAndLowerBoundSieveWeights} is satisfied regardless of the choices for $y_m$.

Following the discussion on Buchstab iterations in \cite[Section 6.2]{Opera}, one may similarly deduce
\begin{equation}
S(\cA, z) = S^+(\cA, z) - \sum_{n \text{ odd}} S_n(\cA, z),
\label{Buchstab-UpperBound}
\end{equation}
\begin{equation}
S(\cA, z) = S^-(\cA, z) + \sum_{n \text{ even}} S_n(\cA, z),
\label{Buchstab-LowerBound}
\end{equation}
where
\begin{equation}
S_n(\cA, z) = \mathop{\sum \cdots \sum}_{ \substack{y_n \preceq \kp_n \prec \cdots \prec \kp_1 \\ \kp_m \prec y_m, \, m < n, \, m \equiv n (2) } } S(\cA_{\kp_1\cdots \kp_n}, \kp_n).
\label{Buchstab-NthError}
\end{equation}
Moreover, by the same procedure,
\begin{equation}
V(z) = V^+(D, z) - \sum_{n \text{ odd}} V_n(z),
\label{Buchstab-UpperBound_Main}
\end{equation}
\begin{equation}
V(z) = V^-(D, z) + \sum_{n \text{ even}} V_n(z),
\label{Buchstab-LowerBound_Main}
\end{equation}
where
\begin{equation}
V_n(z) = \mathop{\sum \cdots \sum}_{ \substack{y_n \preceq \kp_n \prec \cdots \prec \kp_1 \prec z \\ \kp_m \prec y_m, \, m < n, \, m \equiv n (2) } } g(\kp_1\cdots \kp_n) V(\kp_n).
\label{Buchstab-NthError}
\end{equation}
From \eqref{Buchstab-UpperBound} and \eqref{Buchstab-LowerBound}, 
\begin{equation*}
\begin{aligned}
S(\cA, z) & \leq S^+(\cA,z) = X V^+(D,z) + R^+(D,z), \\ 
S(\cA, z) & \geq S^-(\cA,z) = X V^-(D,z) + R^-(D,z).
\end{aligned}
\end{equation*}
Thus, to prove the ``Fundamental Lemma" for a certain choice of truncation parameters $y_m$, it suffices to upper bound $V_n(z)$ in light of \eqref{Buchstab-UpperBound_Main} and \eqref{Buchstab-LowerBound_Main}. 

\subsection{Fundamental Lemma for Zero Dimensional Sieves} We assume the sieve dimension is zero, i.e. $\kappa = 0$ in \eqref{SieveDimension}. 
For the sets defined in \eqref{TruncationSets}, choose the truncation parameters 
\[
y_m = \frac{D}{\N(\kp_1\cdots \kp_m)}
\]
which is an instance of the beta-sieve independently due to Rosser and Iwaniec. Thus, $\lambda^{\pm}_{\kd}$ is a combinatorial weight truncated to $\cD^{\pm}$ with level of support $D$. Define the \emph{sifting variable}
\[
\tau := \frac{\log D}{\log z}. 
\]
As previously remarked, it remains to upper bound $V_n(z)$ as defined in \eqref{Buchstab-NthError}. 

Suppose $n \leq \tau-1$. By our choice of truncation parameters, the condition $y_n \preceq \kp_n$ in \eqref{Buchstab-NthError} implies that $D \leq (\N\kp_1)^{n+1} < z^{n+1} \leq z^{\tau} = D$, a contradiction. Thus,
\[
V_n(z) = 0 \qquad \text{for $n \leq \tau-1$}. 
\]
Now, suppose $n > \tau-1$. Since the terms of $V_n(z)$ are non-negative and $V(\kp_n) \leq 1$, we deduce that 
\begin{align*}
V_n(z) 
& \leq \mathop{\sum \cdots \sum}_{ \substack{\kp_n \prec \cdots \prec \kp_1 \prec z } } g(\kp_1\cdots \kp_n)  
\leq \frac{1}{n!} \Big( \sum_{\kp \prec z} g(\kp)  \Big)^n  
\leq \frac{1}{n!} \big| \log V(z) \big|^n  
\end{align*}
Using \eqref{SieveDimension} with $\kappa = 0$, observe
\[
\frac{V_n(z)}{V(z)} \leq \frac{C (\log C)^n}{n!} \qquad \text{for $n > \tau-1$}. 
\]
Summing over all $n$ of the same parity and using the power series for hyperbolic sine and cosine, observe
\begin{align*}
\sum_{n \text{ odd}} V_n(z) \leq V(z) \cdot \sum_{\substack{ n > \tau-1 \\ n \text{ odd}}} \frac{C (\log C)^n}{n!} = V(z) \cdot \Big[ \frac{C^2-1}{2} - C \sum_{\substack{ 1 \leq n < n_1(\tau) \\ n \text{ odd}}} \frac{ (\log C)^{n }}{n!} \Big], \\
\sum_{n \text{ even}} V_n(z) \leq V(z) \cdot \sum_{\substack{ n > \tau-1 \\ n \text{ even}}} \frac{C (\log C)^n}{n!} = V(z) \cdot \Big[ \frac{C^2+1}{2} - C \sum_{\substack{ 0 \leq n < n_0(\tau) \\ n \text{ even}}} \frac{ (\log C)^{n }}{n!} \Big]
\end{align*}
where $n_1(t)$ is the least odd integer $ > t-1$, and $n_0(t)$ is the least even integer $>t-1$. We have therefore established the following theorem. 
\begin{thm}[Fundamental Lemma for Zero Dimensional Sieves]
\label{FundamentalLemma}
Let $D \geq 1$ and $z \geq 2$. Suppose \eqref{SieveDimension} holds with $\kappa = 0$ for all $w$ with $2 \leq w < z$ and some $C > 1$. Then
\begin{equation}
\begin{aligned}
S(\cA,z) & \leq X V(z) \Big\{ 1 + E_1(C; \tau) \Big\} + R^+(D,z) \\ 
S(\cA,z) & \geq X V(z) \Big\{ 1 - E_0(C; \tau) \Big\} + R^-(D,z) \\ 
\end{aligned}
\end{equation}
where $\tau = \tfrac{\log D}{\log z}$, $n_1(t)$ is the least odd integer $ > t-1$, $n_0(t)$ is the least even integer $>t-1$, 
\begin{align*}
E_1(C; \tau) & =  \frac{C^2-1}{2} - C \sum_{\substack{ 1 \leq n < n_1(\tau) \\ n \text{ odd}}} \frac{ (\log C)^{n }}{n!}, \\
E_0(C; \tau) & = \frac{C^2+1}{2} - C \sum_{\substack{ 0 \leq n < n_0(\tau) \\ n \text{ even}}} \frac{ (\log C)^{n }}{n!},
\end{align*}
and $R^{\pm}(D,z)$ are the remainders given by
\[
R^{\pm}(D,z) =  \sum_{\substack{ \kd \mid \kP(z) \\ \N\kd < D}} \lambda_{\kd}^{\pm} r_{\kd} \qquad \text{with } \quad \text{$|\lambda_{\kd}^{\pm}| \leq 1$}.
\]
\end{thm}
\begin{rem*} Of course, one could replace $E_0(C; \tau)$ and $E_1(C;\tau)$ by simpler expressions using Taylor's theorem but this results in slightly worse constants. 
\end{rem*}
\section{Preliminaries}
\label{sec:Prelim}
\subsection{Elementary Estimates}

Recall Hecke characters are characters $\chi$ of the ray class group $\Cl(\kq)$, writing $\chi \pmod{\kq}$ to indicate this relationship. For notational convenience, we pullback the domain of $\chi$ and extend it to all integral ideals by zero; that is, $\chi(\kn)$ is defined for all integral ideals $\kn \subseteq \cO$ and $\chi(\kn) = 0$ for $(\kn,\kq) \neq 1$. The \emph{conductor} $\fk{f}_{\chi}$ of a Hecke character $\chi \pmod{\kq}$ is the maximal integral ideal such that $\chi$ is the pushforward of a Hecke character modulo $\fk{f}_{\chi}$. Observe $\fk{f}_{\chi}$ divides $\kq$. We say $\chi$ is \emph{primitive modulo $\kq$} if $\kf_{\chi} = \kq$. 

Thus, the Hecke $L$-function associated to $\chi \pmod{\kq}$ may be written as
\[
L(s,\chi) = \sum_{\kn \subseteq \cO} \chi(\kn) (\N{\kn})^{-s} = \prod_{\kp} \Big(1-\frac{\chi(\kp)}{(\N{\kp})^s} \Big)^{-1} \qquad \text{for $\sigma > 1$}
\]
where $s = \sigma + it$. Unless otherwise specified, we may refer to Hecke characters as characters. For completeness, we record a classical convexity bound for Hecke $L$-functions due to Rademacher. 

\begin{lem}[Rademacher] \label{ConvexityBd} Let $\delta \in (0,\tfrac{1}{2})$ be given. Suppose  $\chi$ is a primitive non-principal Hecke character modulo $\kf_{\chi}$. Then for $s = \sigma+it$,
\[
L(s,\chi) \ll_{\delta} \zeta_{\Q}(1+\delta)^{n_K} \Big( \frac{d_K \N\kf_{\chi}}{ (2\pi)^{n_K}} (2+|t|)^{n_K} \Big)^{(1-\sigma+\delta)/2}
\]
and
\[
(s-1) \cdot \zeta_K(s) \ll_{\delta} \zeta_{\Q}(1+\delta)^{n_K} \Big( \frac{d_K}{ (2\pi)^{n_K}} (2+|t|)^{n_K} \Big)^{(1-\sigma+\delta)/2}
\]
uniformly in the region
\[
-\delta \leq \sigma \leq 1+\delta. 
\]  
\end{lem} 
\begin{proof} See \cite[Theorem 5]{Rademacher}.
\end{proof}
When applying the above convexity result, we will require bounds for the Gamma function $\Gamma(s) = \int_0^{\infty} e^{-t} t^{s-1}dt$ in a vertical strip; for instance, from  \cite[Appendix C]{MV},
\begin{equation}
\Gamma(s) \ll_{\delta} e^{-|t|} 
\label{Gamma_VerticalStrip}
\end{equation}
uniformly in the region $-2 \leq \Re\{s\} \leq 2$ with $|s| \geq \delta$. We end this subsection with elementary results involving standard sums over prime ideals and the size of the ray class group $\Cl(\kq)$. 

\begin{lem} \label{PR-NaivePrimeSums}
Let $a \in (0,1), \delta > 0$ be arbitrary and $\kd$ be an integral ideal of $K$. Then
\begin{enumerate}[(i)]
	\item  $\ds\sum_{\kp} \frac{1}{(\N\kp)^{1+\delta}} \ll_{\delta} n_K$
	\item  $\ds\sum_{\kp \mid \kd} \frac{1}{(\N\kp)^{a}} \ll n_K^{a/2} (\log \N\kd)^{1-a/2}$
	\item  $\ds\sum_{\kp \mid \kd} \frac{1}{(\N\kp)^{a}} \ll \delta^{-2/a+1} n_K + \delta \log \N\kd$
\end{enumerate}
\end{lem}
\begin{proof} For (i), observe
\[
\sum_{\kp} \frac{1}{(\N\kp)^{1+\delta}} \leq n_K \sum_{p} \frac{1}{p^{1+\delta}} \ll_{\delta} n_K
\]
where the latter sum is over rational primes $p$.  For (ii), using H\"{o}lder's inequality, we see
\[
\sum_{\kp \mid \kd} \frac{1}{(\N\kp)^{a}} \leq \Big( \sum_{\kp \mid \kd} 1 \Big)^{1-a} \Big(\sum_{\kp \mid \kd} \frac{1}{\N\kp}\Big)^{a}
\]
Bounding the first sum by $\log \N\kd$ and the second sum by the estimate
\[
\sum_{\kp \mid \kd} \frac{\log \N\kp}{\N\kp} \ll n_K^{1/2} (\log \N\kd)^{1/2}
\]
from  \cite[Lemma 2.7]{Zaman2015}, we obtain the desired result. Statement (iii) follows easily from (ii) by considering whether $n_K \leq \delta^{2/a} \log \N\kd$ or not. 
\end{proof}

\begin{lem} \label{PR-RayClassGroup}
Let $\kq$ be an integral ideal. Then $h(\kq) \leq e^{n_K} d_K^{1/2} \N\kq$. 
\end{lem}
\begin{proof} From \cite[p.115]{milneCFT}, one can verify $h(\kq) \leq 2^{n_K} h_K \N\kq$. Then the result follows easily from the fact $h_K \leq d_K^{1/2} (\tfrac{4}{\pi} )^{n_K}$, which is a weaker version of Minkowski's classical bound for the class number. 
\end{proof}

\subsection{Exceptional Character} \label{subsec:ExceptionalCharacter} In this section, we setup notation related to the central object of our study -- the exceptional character $\psi$ -- and subsequently prove various estimates by standard methods. 

Let $\psi \pmod{\kq}$ be a real character with real zero 
\begin{equation}
\beta = 1 - \frac{1}{\eta \log(n_K^{n_K} d_K \N\kq)} \qquad \text{with }\eta \geq 20. 
\label{def:eta}
\end{equation}
For integral ideals $\kn \subseteq \cO$, define
\begin{equation}
\lambda(\kn) := 
\begin{cases} 
\ds  \sum_{\km \mid \kn} \psi(\km) & \text{if $\psi$ is quadratic}, \\
 \chi_0(\kn) & \text{if $\psi$ is principal},
\end{cases}
\label{def:lambda}
\end{equation}
and
\begin{equation}
\rho(\kn) :=  \mu_K^2(\kn) \lambda(\kn)
\label{def:rho}
\end{equation}
 where $\mu_K(\, \cdot \,)$ is defined by \eqref{KMobius} and $\chi_0 \pmod{\kq}$ is the principal Hecke character. First, we collect some simple observations about these functions which we state without proof. 
\begin{lem} \label{LambdaAndRho}
Define $\lambda(\kn)$ and $\rho(\kn)$ as in \eqref{def:lambda} and \eqref{def:rho} respectively. Then:
\begin{enumerate}[(i)]
	\item $\rho(\kn)$ and $\lambda(\kn)$ are multiplicative functions of $\kn$. 
	\item  $\rho(\kp) = \lambda(\kp) = 1$ or $2$ if $\psi(\kp) = 1$ and $\psi$ is principal or quadratic respectively. 
	\item $\rho(\kn) = 0$ if there exists a prime ideal $\kp \mid \kn$ such that $\psi(\kp) = -1$.
	\item $0 \leq \rho(\kn) \leq \lambda(\kn)$
\end{enumerate}
\end{lem}
Next, define 
\begin{equation}
F_{\psi}(s) := \sum_{\kn \subseteq \cO} \frac{\lambda(\kn)}{(\N\kn)^s} \qquad \text{for $\Re\{s\} > 1$.}
\label{def:F_psi}
\end{equation}
We highlight some basic properties of $F_{\psi}(s)$ in the following lemma.
\begin{lem} \label{F_psi-Lemma}Define $F_{\psi}(s)$ as in \eqref{def:F_psi}. Then: 
\begin{enumerate}[(i)]
	\item $F_{\psi}(s)$ extends meromorphically to all of $\C$ with only a simple pole at $s=1$. 
	\item $F_{\psi}(\beta) = 0$ where $\beta$ is the real zero associated to $\psi \pmod{\kq}$. 
	\item For $\delta \in (0,\tfrac{1}{2})$ and $s = \sigma+it$, 
	\[
	F_{\psi}(s) \ll_{\delta} 
	\begin{cases}
	\big\{ d_K^2 (\N\kq) (2+|t|)^{2n_K} \big\}^{(1-\sigma+\delta)/2}  e^{O_{\delta}(n_K)} & \text{if $\psi$ is quadratic},\\ 
	 (\N\kq)^{\delta}  \big\{ d_K (2+|t|)^{n_K} \big\}^{(1-\sigma+\delta)/2} e^{O_{\delta}(n_K)}& \text{if $\psi$ is principal},\\ 
	\end{cases}
	\]
	uniformly in region $\delta \leq \sigma \leq 1+\delta$ with $|s-1| \geq \delta$.	
\end{enumerate}
\end{lem}
\begin{proof} By \eqref{def:lambda},
\[
F_{\psi}(s) =  
\begin{cases} 
\ds L(s,\chi_0) L(s,\psi) & \text{if $\psi$ is quadratic}, \\ 
\ds L(s,\chi_0) & \text{if $\psi$ is principal.}
\end{cases}
\]
From this factorization,  (i) follows from well-known properties of Hecke $L$-functions and (ii) is implied by $L(\beta,\psi) = 0$. For (iii), use \cref{PR-NaivePrimeSums} with $a=\delta$ for the ``imprimitive" part of $F_{\psi}(s)$, i.e. Euler factors corresponding to $\kp \mid \kq$. Then apply \cref{ConvexityBd} to the ``primitive" part and note $\zeta_{\Q}(1+\delta) \ll_{\delta} 1$. 
\end{proof}
In light of \cref{F_psi-Lemma}, we define some naturally-occurring quantities. First, 
\begin{equation}
\label{def:kappa_psi}
\kappa_{\psi} = \Res_{s=1} F_{\psi}(s) = \begin{cases} 
\ds \frac{\varphi_K(\kq)}{\N\kq} \kappa_K L(1,\psi) & \text{if $\psi$ is quadratic}, \\ 
\ds \frac{\varphi_K(\kq)}{\N\kq} \kappa_K & \text{if $\psi$ is principal,}
\end{cases}
\end{equation}
where $\kappa_K$ is the residue of the Dedekind zeta function $\zeta_K(s)$ at $s=1$ and 
\[
\varphi_K(\kq) = \#(\cO/\kq)^{\times} = \N\kq \prod_{\kp \mid \kq} \Big( 1-\frac{1}{\N\kp} \Big)
\]
is the generalized Euler $\varphi$-function. Further, denote
\begin{equation}
d_{\psi} =
\begin{cases} 
\ds d_K^2 \N\kq & \text{if $\psi$ is quadratic}, \\ 
\ds d_K & \text{if $\psi$ is principal,}
\end{cases}
\label{def:d_psi}
\end{equation}
anda
\begin{equation}
W_{\psi} =
\begin{cases} 
\ds n_K^{2n_K} d_K^2 \N\kq  & \text{if $\psi$ is quadratic}, \\ 
\ds n_K^{n_K} d_K & \text{if $\psi$ is principal.}
\end{cases}
\label{def:W_psi}
\end{equation}

For the remainder of this section, we collect various well-known lower and upper bounds for $\kappa_{\psi}$ and establish other relevant estimates involving $\lambda(\kn)$. The arguments are  straightfoward with standard applications of Mellin inversion.
 
\begin{thm}[Stark] \label{kappa_LB_effective}
\[
\frac{1}{\kappa_{\psi}} \ll  \begin{cases}
n_K^{2n_K} d_K^{1/n_K} (\N\kq)^{1/2n_K} & \text{if $\psi$ is quadratic}, \\
n_K^{n_K} d_K^{1/n_K} & \text{if $\psi$ is principal.}
 \end{cases}
\]
where all implicit constants effective. 
\end{thm}
\begin{proof} This is a rephrasing of \cite[Theorem 1]{Stark} to our context. To be clear, if $\psi$ is principal then $\kappa_{\psi} \geq \kappa_K$ so the result follows from \cite[Theorem 1]{Stark}. If $\psi$ is quadratic, then consider the quadratic extension of $K$ given by $M = K(\psi)$. It follows that $\kappa_{\psi} \geq \kappa_M$ so we can once again apply \cite[Theorem 1]{Stark} to obtain the desired bound. 
\end{proof} 
\begin{thm}[Brauer-Siegel] \label{kappa_LB_ineffective}
For $\delta > 0$, 
\[
\frac{1}{\kappa_{\psi}} \ll_{\delta} d_{\psi}^{\delta}
\]
where the implicit constant is ineffective. 
\end{thm}
\begin{proof} Similar to \cref{kappa_LB_effective} instead of using \cite[Theorem 1]{Stark} to bound the residues of Dedekind zeta functions, we apply the celebrated Brauer-Siegel theorem:  
\[
\kappa_M \gg_{\epsilon} d_M^{-\epsilon}
\]
for any number field $M$, where the implicit constant is ineffective. See \cite{Brauer} for details. 
\end{proof}
\cref{kappa_LB_ineffective} is the only result with ineffective constants so unless otherwise stated, \emph{all implicit constants are effective and absolute}. 

 \begin{lem}[Stark]  \label{kappa_LB_zero}  $\kappa_{\psi} \gg 1-\beta$. 
 \end{lem}
\begin{proof} This is again an analogous rephrasing of \cite[Lemma 4]{Stark} to our context. 
\end{proof}

\begin{lem} \label{Lambda_InitialSum} For $\delta > 0$, 
\[
\sum_{\kn } \frac{\lambda(\kn)}{(\N\kn)^{\beta}} e^{-\N\kn/y} =  \kappa_{\psi} \Gamma(1-\beta) \big\{ 1 + O(\delta) + O\big( (1-\beta) \log y \big) \big\}
\]
provided
\begin{equation}
\label{Lambda_InitialSum-yRange}
y \geq W_{\psi}^{1/2+\delta} (\N\kq)^{\delta} e^{M_{\delta}n_K}
\end{equation}
for some sufficiently large constant $M_{\delta} \geq 1$. 
\end{lem}
\begin{proof}  For the upper bound, apply Mellin inversion to see
\[
S :=  \sum_{\kn} \frac{\lambda(\kn)}{(\N\kn)^{\beta}}e^{- \N\kn/y} = \frac{1}{2\pi i} \int_{2-i\infty}^{2+i\infty} F_{\psi}(s+\beta) \Gamma(s) y^s ds.
\]
Shift the line of integration to $\Re\{s\} = 1/2-\beta$, pick up the pole $s=1-\beta$, and bound the remaining integral  using \cref{F_psi-Lemma}(iii) and  \eqref{Gamma_VerticalStrip}.  Therefore,
\[
S =\Big\{ \kappa_{\psi} \Gamma(1-\beta) + O_{\delta}\Big( \frac{W_{\psi}^{1/4+\delta/2} (\N\kq)^{\delta/2}  e^{O_{\delta}(n_K)}  }{y^{1/2}}  \Big) \Big\} y^{1-\beta}.
\]
From \cref{kappa_LB_zero} and condition \eqref{Lambda_InitialSum-yRange}, it follows that the main term dominates the error provided $M_{\delta}$ is sufficiently large. This yields the desired result upon writing $y^{1-\beta} = 1 + O((1-\beta)\log y)$. 
\end{proof}

\begin{lem} \label{Lambda_ShortSum} For $\delta > 0$ and $y_2 \geq 3y_1$, 
\[
\sum_{\kn} \frac{\lambda(\kn)}{\N\kn}\big( e^{-\N\kn/y_2} - e^{-\N\kn/y_1}\big) \ll \kappa_{\psi} \log(y_2/y_1)
\]
provided 
\[
y_1 \geq  \kappa_{\psi}^{-1-\delta} W_{\psi}^{1/2+\delta} (\N\kq)^{\delta} e^{M_{\delta}n_K}
\]
for some sufficiently large constant $M_{\delta} \geq 1$. 
\end{lem}
\begin{proof} By Mellin inversion,
\[
S' := \sum_{\kn} \frac{\lambda(\kn)}{\N\kn}\big( e^{-\N\kn/y_2} - e^{-\N\kn/y_1}\big) = \frac{1}{2\pi i} \int_{2-i\infty}^{2+i\infty} F_{\psi}(s+1) \Gamma(s) \big\{ y_2^s - y_1^s\} ds.
\]
Shift the line of integration to $\Re\{s\} = -1+\delta$,  pick up the simple pole at $s=0$, and bound the remaining integral  using \cref{F_psi-Lemma}(iii) and  \eqref{Gamma_VerticalStrip}.  Thus, for $\delta > 0$, 
\[
S' = \kappa_{\psi} \log(y_2/y_1) + O_{\delta}\Big( \frac{W_{\psi}^{1/2+\delta/2} (\N\kq)^{\delta/2} e^{O_{\delta}(n_K)} }{y_1^{1-\delta}}  \Big).
\]
Since $\log(y_2/y_1) \gg 1$, the result follows from the condition on $y_1$. 
\end{proof}

\section{Application of the sieve}
\label{sec:SieveApplication}
\subsection{Sieve sequence}
\label{subsec:SieveSequence}
Fix a ray class $\cC \in \Cl(\kq)$ satisfying $\psi(\cC) = 1$ and retain the notation of \cref{subsec:ExceptionalCharacter}. Recall the Hecke $L$-function $L(s,\psi)$ is assumed to have a real zero 
\[
\beta = 1 - \frac{1}{\eta \log( n_K^{n_K} d_K\N\kq)}
\]
with $\eta \geq 20$.  Let $2 \leq z \leq x$. During the course of our arguments, the parameter $z$ will be chosen and the valid range of $x$ will be specified. We wish to apply the sieve to the sequence 
\begin{equation}
\cA = \cA(x) = \{ a_{\kn} \}_{\kn \subseteq \cO} 
\quad \text{ with } \quad a_{\kn} = \rho(\kn) e^{-\N\kn/x} \cdot  \mathbf{1}\{\kn \in \cC\}
\label{AS-SieveSequence}
\end{equation}
where $\rho(\kn)$ is defined in \eqref{def:rho} and $\mathbf{1}\{ \, \cdot \,\}$ is an indicator function. Choose the set of prime ideals to be
\begin{equation}
\cP = \{ \kp \subseteq \cO \text{ prime} : \psi(\kp) = 1\}
\label{AS-SievePrimes}
\end{equation}
and denote
\begin{equation}
\cD = \{ \kd \subseteq \cO \text{ square-free}: \kp \mid \kd \implies \psi(\kp) = 1\}.
\label{AS-SieveIdeals}
\end{equation}
\subsection{Local Densities}
\begin{lem} \label{AS-LocalDensities} Let $\kd \in \cD$. Then, for any $\delta > 0$, 
\[
|\cA_{\kd}| = \sum_{\substack{\kn \in \cC \\ \kd \mid \kn} } \rho(\kn) e^{-\N\kn/x}  = g(\kd) X +  r_{\kd} 
\]
with $X = b_{\psi} \kappa_{\psi}  \cdot \dfrac{x}{h(\kq)}$, where if $\psi$ is quadratic then
\begin{equation*}
\begin{aligned}
b_{\psi}&  = 2 \prod_{\psi(\kp) = 1} \Big(1-\frac{3}{\N\kp^2} + \frac{2}{\N\kp^3} \Big) \prod_{\psi(\kp)=-1} \Big( 1 - \frac{1}{\N\kp^2} \Big), \\
g(\kp)  & = \frac{2}{\N\kp +2} \quad \text{for $\kp \in \cP$}, \\
|r_{\kd}| & \ll \frac{x^{1/2+\delta}}{(\N\kd)^{1/2}} \cdot ( n_K^{n_K}  d_K \N\kq)^{(1+\delta)/2} \cdot e^{O_{\delta}(n_K)},
\end{aligned}
\end{equation*}
and if $\psi$ is principal then
\begin{equation*}
\begin{aligned}
b_{\psi}&  = \prod_{\kp \nmid \kq} \Big( 1 - \frac{1}{\N\kp^2} \Big), \\
g(\kp)  & = \frac{1}{\N\kp +1} \quad \text{for $\kp \in \cP$}, \\
|r_{\kd}| & \ll \frac{x^{1/2+\delta}}{(\N\kd)^{1/2}} \cdot ( n_K^{n_K}  d_K \N\kq)^{(1+\delta)/4} \cdot e^{O_{\delta}(n_K)}. 
\end{aligned}
\end{equation*}
\end{lem} 
\begin{rem*} If $\kd \not\in \cD$, then $|A_{\kd}| = 0$ by \cref{LambdaAndRho}. Thus, for prime ideals $\kp \not\in \cP$, set $g(\kp) = 0$ and multiplicatively extend the function $g$  to all integral ideals of $\cO$. 
\end{rem*}
\begin{proof} We adapt the proof of \cite[Lemma 1]{HBSiegel} with some modifications when bounding the remainder terms $r_{\kd}$. Write
\[
f(s,\chi) := \sum_{ \substack{ \kn \subseteq \cO  \\ \kd \mid \kn} } \rho(\kn) \chi(\kn) (\N\kn)^{-s} \qquad \text{for $\Re\{s\} > 1$}
\]
so, by orthogonality and Mellin inversion, 
\begin{equation}
\begin{aligned}
\sum_{\substack{\kn \in \cC \\ \kd \mid \kn} } \rho(\kn) e^{-\N\kn/x}   
& = \frac{1}{h(\kq)} \sum_{\chi \pmod{\kq} } \bar{\chi}(\cC) \frac{1}{2\pi i} \int_{2 -i\infty}^{2+i\infty} f(s,\chi) \Gamma(s) x ^s ds. 
\end{aligned}
\label{S1Lemma1}
\end{equation}
Alternatively, we may write $f(s,\chi)$ as an Euler product to see that
\[
f(s,\chi) = \rho(\kd) \chi(\kd) (\N\kd)^{-s}  \times  \prod_{ \substack{ \kp \nmid \kd \\ \psi(\kp) = 1} }  \Big( 1 + \rho(\kp)  \frac{ \chi(\kp)}{(\N\kp)^{s}} \Big) \times  \prod_{ \substack{ \kp \nmid \kd \\ \psi(\kp) = -1} } 1.
\]
Note that prime ideals $\kp \mid \kd$ do not appear in the Euler product since $\rho(\kn) = 0 $ for $\kn$ not square-free. Including these analogous factors, we may write
\begin{equation}
f(s,\chi) = \prod_{ \substack{ \psi(\kp) = 1} }  \Big( 1 + \rho(\kp) \frac{ \chi(\kp)}{(\N\kp)^{s}} \Big) \times  \prod_{ \substack{ \psi(\kp) = -1} } 1 \times g_{\kd}(s,\chi) 
\label{f-EulerProd}
\end{equation}
where
\[
g_{\kd}(s,\chi) =  \rho(\kd) \chi(\kd) (\N\kd)^{-s}  \prod_{ \substack{\kp \mid \kd \\ \psi(\kp) = 1} }  \Big( 1 +  \rho(\kp) \frac{ \chi(\kp)}{(\N\kp)^{s}}  \Big)^{-1}.
\]
On the other hand, 
\begin{align*}
L(s,\chi) L(s, \chi\psi) & = \prod_{ \substack{ \psi(\kp) = 1} }  \Big( 1 -  2 \frac{\chi(\kp) }{(\N\kp)^{s}}  + \frac{\chi^2(\kp) }{(\N\kp)^{2s}} \Big)^{-1}  \times  \prod_{ \substack{ \psi(\kp) = -1} } \Big(1 -\frac{ \chi^2(\kp)}{ (\N\kp)^{2s}} \Big)^{-1}, \\
L(s,\chi) & = \prod_{\kp \nmid \kq} \Big( 1 - \frac{\chi(\kp)}{(\N\kp)^s} \Big)^{-1}. \\
\end{align*}
Upon comparing with \eqref{f-EulerProd}, we deduce
\begin{equation}
f(s,\chi) = 
g_{\kd}(s,\chi) g(s,\chi) G(s,\chi) 
\label{f-Factor}
\end{equation}
where
\begin{align*}
g(s,\chi) 
& =  
\begin{cases} 
\ds \prod_{ \substack{ \psi(\kp) = 1} }  \Big( 1 -  3 \frac{\chi^2(\kp) }{(\N\kp)^{2s}}  + 2\frac{\chi^3(\kp)}{ (\N\kp)^{3s} }\Big) \times  \prod_{ \substack{ \psi(\kp) = -1} } \Big(1 - \frac{\chi^2(\kp) }{(\N\kp)^{2s}} \Big) &\text{if $\psi$ is quadratic,}\\
\ds \prod_{\kp \nmid \kq} \Big(1-\frac{\chi^2(\kp)}{(\N\kp)^{2s}}\Big)  & \text{if $\psi$ is principal,}
\end{cases} \\\\
G(s,\chi) 
& =\begin{cases}  L(s,\chi) L(s,\chi\psi) & \text{if $\psi$ is quadratic,}\\\\
L(s,\chi) & \text{if $\psi$ is principal.}
\end{cases}\\
\end{align*}

Therefore, $f(s,\chi)$ has meromorphic continuation to $\C$ and is analytic in $\Re\{s\} > 1/2$, except possibly for a pole at $s=1$ when $\chi$ or $\chi\psi$ is principal. 

Furthermore, we claim
\begin{align}
\label{GrowthGd}
g_{\kd}(s,\chi) & \ll_{\delta} e^{O_{\delta}(n_K)} (\N\kd)^{-1/2}, \\
g(s,\chi)  & \ll_{\delta} e^{O_{\delta}(n_K)},
\label{GrowthG}
\end{align}
uniformly in the region $\Re(s) \geq 1/2+\delta$  for any $\delta > 0$. Here we ignore $s$ in neighborhoods of poles arising from local factors  of $g_{\kd}(s,\chi)$ with $\N\kp < 4$. To see the claim, notice \eqref{GrowthG} follows from \cref{PR-NaivePrimeSums}(i). Estimate \eqref{GrowthGd} follows from \cref{PR-NaivePrimeSums}(iii) with $a=1/2$ combined with the observation
\[
\rho(\kd) \ll \sum_{\kp \mid \kd}  1  \ll \log \N\kd \ll_{\delta} (\N\kd)^{\delta},
\]
thus proving the claim.

Now, we move the line of integration in \eqref{S1Lemma1} from $\Re\{s\} = 2$ to $\Re\{s\} = \tfrac{1}{2} + \delta$. This yields a main term of
\[
R =   \frac{x}{h(\kq)} \sum_{\chi \pmod{\kq} } \bar{\chi}(\cC)  \Res_{s=1} f(s,\chi) . 
\] 
Before computing $R$, observe that since $\psi(\cC) = 1$ and $\psi^2 = \chi_0$
\begin{align*}
 f(s,\chi_0) = f(s,\psi),  \qquad &
G(s,\chi_0) = F_{\psi}(s), \\
 g(1,\chi_0) = g(1,\psi), \qquad & 
g_{\kd}(1,\chi_0) = g_{\kd}(1,\psi) =  g(\kd), 
\end{align*}
where $F_{\psi}(s)$  and $g(\kd)$ are defined in \eqref{def:F_psi} and the statement of \cref{AS-LocalDensities} respectively. Therefore, if $\psi$ is quadratic, the main term $R$ picks up residues for $\chi = \chi_0$ and $\chi = \psi$. Namely,
\begin{align*}
R  & = \frac{x}{h(\kq)} \Big[ \Res_{s=1} f(s,\chi_0) + \bar{\psi}(\cC) \Res_{s=1} f(s, \psi) \Big] \\
 & = \frac{x}{h(\kq)} \cdot g(1, \chi_0) g(\kd)  \cdot \Big[ 2\Res_{s=1} F_{\psi}(s) \Big] \\
 & =  \frac{x}{h(\kq) } \cdot g(1,\chi_0) g(\kd) \cdot 2\kappa_{\psi} \\
 & = g(\kd) X
\end{align*}
since $b_{\psi} = 2 g(1,\chi_0)$ when $\psi$ is quadratic.  If $\psi$ is principal, the main term $R$ picks up a residue for $\chi = \chi_0$ only. In other words,  
\begin{align*}
R  & = \frac{x}{h(\kq)} \cdot \Res_{s=1} f(s,\chi_0)\\
 & = \frac{x}{h(\kq)} \cdot g(1,\chi_0) g(\kd)  \cdot \Res_{s=1} L(s,\chi_0) \\
 & =\frac{x}{h(\kq)} \cdot g(1,\chi_0) g(\kd)  \cdot \kappa_{\psi} \\
 & = g(\kd) X 
\end{align*}
since $b_{\psi} = g(1,\chi_0)$ when $\psi$ is principal. 

Thus far, we have shown
\[
|\cA_{\kd}| = g(\kd) X + r_{\kd}
\]
where 
\[
r_{\kd} = \frac{1}{h(\kq)} \sum_{\chi \pmod{\kq} } \bar{\chi}(\cC) \frac{1}{2\pi i} \int_{1/2+\delta-i\infty}^{1/2+\delta+i\infty} f(s,\chi) \Gamma(s) x ^s ds. 
\]
To bound the remainder, we factor $f(s,\chi)$ via \eqref{f-Factor} and apply the estimates \eqref{GrowthGd}, \eqref{GrowthG}, and  \eqref{Gamma_VerticalStrip}. This yields
\[
|r_{\kd}| \ll_{\delta}  \frac{x^{\tfrac{1}{2}+\delta} e^{O_{\delta}(n_K)}}{h(\kq) (\N\kd)^{1/2}} \sum_{\chi \pmod{\kq} } \int_{0}^{\infty} |G(\tfrac{1}{2} + \delta + it ,\chi) | e^{-|t|}dt
\]
so the desired result then follows from the convexity bound for Hecke $L$-functions (\cref{ConvexityBd}).  
\end{proof}

Motivated by the bounds on the remainder terms $r_{\kd}$ in \cref{AS-LocalDensities}, we define
\begin{equation}
Q_{\psi} = \begin{cases}  
(n_K^{n_K} d_K \N\kq)^{1/2} & \text{if $\psi$ is quadratic},  \\
(n_K^{n_K} d_K \N\kq)^{1/4} & \text{if $\psi$ is principal},
\end{cases}
\label{def:Q_psi}
\end{equation}
so more simply 
\[
|r_{\kd}| \ll \frac{x^{1/2+\delta}}{(\N\kd)^{1/2}}Q_{\psi}^{1+\delta} e^{O_{\delta}(n_K)}.
\]

\subsection{Sieve Dimension}
We prove our sieve problem is zero-dimensional. 


\begin{lem} \label{AS-PsiSmallPrimes}  For $\delta > 0$, 
\[
\sum_{\substack{ \N\kp < z \\ \psi(\kp) = 1 } } \frac{1}{\N\kp} \leq 1+ \delta
\]
provided $\eta \geq \eta(\delta)$ and $z \leq  (n_K^{n_K} d_K \N\kq)^{O_{\delta}(1)}$. 
\end{lem}
\begin{proof} According to \cref{Lambda_InitialSum}, set
\[
y = C_{\delta} W_{\psi}^{1/2+\delta} (\N\kq)^{\delta} e^{M_{\delta}n_K}
\]
where $W_{\psi}$ is defined in \eqref{def:W_psi}.  Using $\lambda(\kn)$ defined in \eqref{def:lambda} and its properties described in \cref{LambdaAndRho}, one can verify that $\lambda(\kn) \leq \lambda(\kn\kp)$ for $\psi(\kp) = 1$ and $\kn \subseteq \cO$ and so
\begin{equation}
\Big(\sum_{\substack{ \N\kp < z \\ \psi(\kp) = 1} } \frac{1}{(\N\kp)^{\beta}}\Big)\Big( \sum_{\kn} \frac{\lambda(\kn)}{(\N\kn)^{\beta}}e^{-\N\kn/y} \Big) \leq \sum_{\kn} \frac{\lambda(\kn)}{(\N\kn)^{\beta}}e^{-\N\kn/yz} 
\label{AS-PsiSmallPrimes_Obsv}
\end{equation}
which we write as $S_1S_2 \leq S_3$, say. It suffices to show $S_1 \leq 1+\delta$. By our choice $y$, we may apply \cref{Lambda_InitialSum} to $S_2$ and $S_3$ deducing
\[
S_1 \leq 1 + \delta + O((1-\beta) \log(yz)).
\]
Since $yz \leq (n_K^{n_K} d_K\N\kq)^{O_{\delta}(1)}$ by our assumption on $z$ and choice of $y$, we conclude
\[
S_1 \leq  1 + \delta + O_{\delta}(\eta^{-1})
\]
whence the result follows after rescaling $\delta$. 
\end{proof}
\begin{cor} \label{AS-Dimension}
Let $g(\kd)$ be the multiplicative function defined in \cref{AS-LocalDensities} and $\delta > 0$ be arbitrary. Then, provided $\eta \geq \eta(\delta)$, 
\[
\frac{V(w)}{V(z)} = \prod_{w \leq \N\kp < z} \Big(1-g(\kp) \Big)^{-1}\leq C_{\psi} := 
\begin{cases} 
e^{2+\delta} & \text{if $\psi$ is quadratic}, \\
 e^{1+\delta} & \text{if $\psi$ is principal},
\end{cases}
\]
for all $2 \leq w \leq z \leq (n_K^{n_K} d_K \N\kq)^{O_{\delta}(1)}$. In particular, \eqref{SieveDimension} holds with $C = C_{\psi}$ and $\kappa = 0$. 
\end{cor}
\subsection{Small Prime Ideal Factors}

With the local densities and dimension computed, we may now apply the ``Fundamental Lemma" and sieve out small primes.  Before doing so, we restrict the choice of sieve parameters for the remainder of the section. For $\delta > 0$, suppose
\begin{equation}
\{\kappa_{\psi}^{-1}+1\}^{1+\delta}  \cdot W_{\psi}^{1/2+\delta} (\N\kq)^{\delta} e^{M_{\delta}n_K} \leq z \leq (n_K^{n_K} d_K\N\kq)^{O_{\delta}(1)}
\label{AS-SieveParameters_z}
\end{equation}
for some sufficiently large constant $M_{\delta}$ and define
\begin{equation}
D = \frac{x^{1-4\delta}}{h(\kq)^2 Q_{\psi}^{2+2\delta}}, \qquad \tau = \frac{\log D}{\log z},
\label{AS-SieveParameters}
\end{equation}
where  $\kappa_{\psi}, W_{\psi}, Q_{\psi}$ are defined in \eqref{def:kappa_psi},  \eqref{def:W_psi}, and \eqref{def:Q_psi} respectively. 

\begin{prop}
\label{AS-SmallPrimes} 
For $\delta > 0$, suppose the sifting level $z$ satisfies \eqref{AS-SieveParameters_z} and define the level of distribution $D$ and sifting variable $\tau$ as in \eqref{AS-SieveParameters}. Assume $\eta \geq \eta(\delta)$ and
\begin{equation}
x \geq  \big\{ \kappa_{\psi}^{-1} W_{\psi}^{1/4} Q_{\psi} \cdot h(\kq) \big\}^{2+20\delta} e^{O_{\delta}(n_K)}.
\label{AS-SmallPrimes_xRange}
\end{equation}
Then
\begin{equation}
\begin{aligned}
S(\cA, z) & \leq XV(z)\Big\{1 + E_1(C_{\psi}; \tau) + O_{\delta}\big( \frac{1}{\log x} \big) \Big\}, \\
S(\cA, z) & \geq XV(z)\Big\{ 1 - E_0(C_{\psi}; \tau) + O_{\delta}\big( \frac{1}{\log x} \big) \Big\},
\end{aligned}
\label{AS-S(A,z)-EQ}
\end{equation}
where $E_0$ and $E_1$ are defined in \cref{FundamentalLemma}, and $C_{\psi}$ is defined in \cref{AS-Dimension}.
\end{prop}
\begin{proof} We only prove the lower bound; the upper bound follows similarly. With the described choice of parameters, we employ the Fundamental Lemma for zero-dimensional sieves (\cref{FundamentalLemma}) in conjunction with \cref{AS-LocalDensities,AS-Dimension}, yielding 
\begin{equation}
\label{AS-SmallPrimes-Step1}
S(\cA,z) \geq XV(z) \Big\{1 - E_0(C_{\psi}; \tau)  \Big\} + R^-(\cA,D).
\end{equation}
Since the sequence $\cA = \{a_{\kn}\}_{\kn}$ is only supported on the set $\cD$ (defined in \cref{subsec:SieveSequence}), 
\[
R^-(\cA,D) \ll \sum_{\substack{\N\kd < D \\ \kd \in \cD}} |r_{\kd}|. 
\]
From \cref{LambdaAndRho}, it follows $\mathbf{1}\{\kd \in \cD\} \leq \lambda(\kd)$ so by \cref{AS-LocalDensities},
\begin{equation}
\begin{aligned}
R^-(\cA,D) & \ll   x^{1/2+\delta} Q_{\psi}^{1+\delta} e^{O_{\delta}(n_K)}\sum_{\substack{\N\kd < D}} \lambda(\kd) (\N\kd)^{-1/2} \\
& \ll   x^{1/2+\delta} Q_{\psi}^{1+\delta} e^{O_{\delta}(n_K)}\sum_{\kd} \lambda(\kd) (\N\kd)^{-1/2} e^{-\N\kd/D}.
\end{aligned}
\label{AS-SmallPrimes_Remainder}
\end{equation}
By Mellin inversion, the sum over $\kd$ equals
\[
\frac{1}{2\pi i} \int_{1-i\infty}^{1+i\infty} F_{\psi}(s+\tfrac{1}{2}) \Gamma(s) D^s ds. 
\]
Pulling the contour to $\Re\{s\} = \delta$, we pick up a main term of $\kappa_{\psi} \Gamma(1/2) D^{1/2}$ and bound the resulting integral using \cref{F_psi-Lemma} and  \eqref{Gamma_VerticalStrip}. Applying these estimates in \eqref{AS-SmallPrimes_Remainder}, we find
\[
R^{-}(\cA,D) \ll_{\delta} x^{1/2+\delta} Q_{\psi}^{1+\delta} e^{O_{\delta}(n_K)}\Big(\kappa_{\psi} D^{1/2} + W_{\psi}^{1/4+\delta}  (\N\kq)^{\delta}  e^{O_{\delta}(n_K)} D^{\delta} \Big) 
\]
By \eqref{AS-SmallPrimes_xRange}, the first term in the parentheses dominates whence
\[
R^{-}(\cA,D) \ll_{\delta} \kappa_{\psi} Q_{\psi}^{1+\delta}   x^{1/2+\delta} D^{1/2} e^{O_{\delta}(n_K)} \ll_{\delta} \frac{\kappa_{\psi} x^{1-\delta}}{h(\kq)}e^{O_{\delta}(n_K)}. 
\]
Since $z$ satisfies the upper bound in \eqref{AS-SieveParameters_z},  it follows from \cref{AS-Dimension} and the definition of $X$ in \cref{AS-LocalDensities} that
\[
XV(z) \gg \frac{\kappa_{\psi} x}{h(\kq)} \cdot \frac{1}{e^{O_{\delta}(n_K)}}
\]
for $\eta \geq \eta(\delta)$ so by these two observations, we conclude
\[
R^{-}(\cA,D) \ll_{\delta} XV(z) x^{-\delta/2} \ll_{\delta} \frac{XV(z)}{\log x}
\]
provided $x \geq e^{O_{\delta}(n_K)}$.  This latter condition on $x$ is clearly implied by assumption \eqref{AS-SmallPrimes_xRange}. Substituting this estimate into \eqref{AS-SmallPrimes-Step1} yields the desired result. 
\end{proof}
\subsection{Large Prime Ideal Factors}

\begin{lem}
\label{AS-LargePrimes}  Suppose $\kp \in \cP$ satisfies $z \leq \N\kp < x^{1/2}$ and assume $z$ satisfies \eqref{AS-SieveParameters_z}. Then for $\delta > 0$,
\[
S(\cA_{\kp}, \kp) \ll_{\delta} \frac{XV(z)}{\N\kp} 
\]
provided $\eta \geq \eta(\delta)$ and
\begin{equation}
x \geq  \big\{ \kappa_{\psi}^{-1} W_{\psi}^{1/4} Q_{\psi} \cdot h(\kq) \big\}^{4+50\delta}  e^{O_{\delta}(n_K)}.
\label{AS-LargePrimes_xRange}
\end{equation}
\end{lem}
\begin{proof} From \cref{sec:Sieve_NT}, recall $S(\cA_{\kp}, \kp) \leq S(\cA_{\kp}, \N\kp)$ so it suffices to bound the latter. Using \cref{AS-LocalDensities,AS-Dimension}, we apply the upper bound sieve from \cref{FundamentalLemma} to the sequence $\cA_{\kp}$ with level of distribution $D' = D/\N\kp$, sifting level $z' = \N\kp$, and sifting variable $\tau' = \frac{\log D'}{\log z'}$.  This application therefore yields
\[
S(\cA_{\kp}, \N\kp) \ll g(\kp) X V(z) + \sum_{\substack{ \kd \mid \kP(z') \\ \N\kd < D'} } |r_{\kp\kd}|
\]
since $V(\N\kp) \leq V(z)$ for $\N\kp \geq z$. 
As $g(\kp) \ll (\N\kp)^{-1}$ by \cref{AS-LocalDensities}, it suffices to bound the remainder sum. Following the same argument as in \cref{AS-SmallPrimes}, we see
\begin{align*}
\sum_{\substack{ \kd \mid \kP(z') \\ \N\kd < D'} } |r_{\kp\kd}|
& \ll_{\delta} \frac{x^{1/2+\delta}}{(\N\kp)^{1/2}} Q_{\psi}^{1+\delta} e^{O_{\delta}(n_K)}\sum_{\kd} \frac{\lambda(\kd)}{(\N\kd)^{1/2}} e^{-\N\kd\kp/D} \\
& \ll_{\delta} \frac{x^{1/2+\delta}}{(\N\kp)^{1/2}} Q_{\psi}^{1+\delta} e^{O_{\delta}(n_K)} \Big(\kappa_{\psi} \big(\frac{D}{\N\kp}\big)^{1/2} + W_{\psi}^{1/4+\delta}  (\N\kq)^{\delta}  e^{O_{\delta}(n_K)} \big(\frac{D}{\N\kp}\big)^{\delta} \Big)  \\
& \ll_{\delta} \frac{1}{\N\kp} \cdot \kappa_{\psi} x^{1/2+\delta} Q_{\psi}^{1+\delta} D^{1/2}e^{O_{\delta}(n_K)}
\end{align*}
provided \eqref{AS-LargePrimes_xRange} holds. One can similarly show that the above is $\ll_{\delta} XV(z) (\N\kp)^{-1}$ since $z$ satisfies the upper bound in \eqref{AS-SieveParameters_z} and $\eta \geq \eta(\delta)$.
\end{proof}

\begin{lem} \label{AS-PsiLargePrimes}  Let $\delta > 0$ and assume $z$ satisfies \eqref{AS-SieveParameters_z}. For $x > 2z$, 
\[
\sum_{\substack{ z \leq \N\kp < x \\ \psi(\kp) = 1} } \frac{1}{\N\kp} \ll_{\delta}  (1-\beta) \log x
\]
provided $\eta \geq \eta(\delta)$. 
\end{lem}
\begin{proof} From \cref{LambdaAndRho} and the condition $x > 2z$, notice
\[
\sum_{\substack{ z \leq \N\kp < x \\ \psi(\kp) = 1} } \frac{1}{\N\kp} \ll \sum_{z \leq \N\kp < x} \frac{\lambda(\kp)}{\N\kp} \big\{ e^{-\N\kp/x} - e^{-\N\kp/z} \big\} = S_1,
\]
say, so we estimate $S_1$. Observe that
\[
S_2 := \sum_{\kn} \frac{\lambda(\kn)}{(\N\kn)^{\beta}} e^{-2\N\kn/z} = \sum_{\kn} \frac{\lambda(\kn)}{\N\kn} e^{-\N\kn/z} \cdot H_{1-\beta}\Big( \frac{\N\kn}{z} \Big) z^{1-\beta}
\]
where $H_{\epsilon}(t) = t^{\epsilon} e^{-t}$ for $\epsilon > 0$ and $t > 0$. By calculus, $H_{\epsilon}(t)$ is maximized at $t=\epsilon$ and $H_{\epsilon}(\epsilon) \rightarrow 1$ as $\epsilon \rightarrow 0^+$. Moreover, $z^{1-\beta} = 1 + O((1-\beta) \log z)$. Therefore, by \eqref{AS-SieveParameters_z},
\[
S_2 \ll_{\delta} \sum_{\kn} \frac{\lambda(\kn)}{\N\kn} e^{-\N\kn/z}
\]
for $\eta \geq \eta(\delta)$. Hence, using \cref{LambdaAndRho}, we see
\[
\begin{aligned}
S_1S_2 \ll & \Big( \sum_{z \leq \N\kp < x}  \frac{\lambda(\kp)}{\N\kp} \big\{ e^{-\N\kp/x} - e^{-\N\kp/z} \big\} \Big)
\Big( \sum_{\kn} \frac{\lambda(\kn)}{\N\kn} e^{-\N\kn/z} \Big) \\
& \qquad \qquad \ll \sum_{\kn} \frac{\lambda(\kn)}{\N\kn}\big\{ e^{-\N\kn/xz} - e^{-\N\kn/z} \big\} = S_3,
\end{aligned}
\]
say. By both the lower and upper bound of \eqref{AS-SieveParameters_z}, we may lower bound $S_2$ using \cref{Lambda_InitialSum} and upper bound $S_3$ using \cref{Lambda_ShortSum}. Combining these estimates yields the desired bound for $S_1$ for $\eta \geq \eta(\delta)$.  
\end{proof}
\section{Proof of \cref{Theorem1}}
\label{sec:TheoremProof}
We claim \cref{Theorem1} is a consequence of the following result.
\begin{thm} \label{T1-MainTheorem}  Suppose $\psi \pmod{\kq}$ is an real Hecke character of the number field $K$ with associated real zero $\beta$ as in \eqref{RealZero}. Let $\cC \in \Cl(\kq)$ satisfy $\psi(\cC) = 1$ and $\delta > 0$ be given. Denote $X$ as per \cref{AS-LocalDensities}. Assume $x$ satisfies both of the following
\begin{align}
\label{T1-xRange_U}
x &  \leq (n_K^{n_K} d_K \N\kq)^{100} e^{M_{\delta}n_K}, \\
\label{T1-xRange_L1}
x & \geq \{ (\kappa_{\psi}^{-1}+1)^{4} W_{\psi} Q_{\psi}^4  h(\kq)^4 \}^{1+50\delta} e^{M_{\delta}n_K},
\end{align}
for $M_{\delta} > 0$ sufficiently large. If $\psi$ is quadratic then
\begin{align*}
& \sum_{\substack{ \N\kp < x \\ \kp \in \cC}} \rho(\kp)  \geq 0.00466 \cdot X
\end{align*}
provided $\eta \geq \eta(\delta)$ and additionally
\begin{align}
\label{T1-xRange_L2-Quadratic}
x & \geq  \{ (\kappa_{\psi}^{-1}+1)^{5} W_{\psi}^{5/2} Q_{\psi}^2 h(\kq)^2 \}^{1+50\delta} e^{M_{\delta}n_K}.
\end{align}
Otherwise, if $\psi$ is principal then
\begin{align*}
& \sum_{\substack{ \N\kp < x \\ \kp \in \cC}} \rho(\kp)  \geq 0.0557 \cdot X
\end{align*}
provided $\eta \geq \eta(\delta)$ and additionally
\begin{align}
\label{T1-xRange_L2-Principal}
x & \geq  \{ (\kappa_{\psi}^{-1}+1)^3 W_{\psi}^{3/2} Q_{\psi}^2 h(\kq)^2 \}^{1+50\delta} e^{M_{\delta}n_K}.
\end{align}
\end{thm}
\begin{rem*} Recall $\kappa_{\psi}, W_{\psi}, Q_{\psi}$ are defined in \eqref{def:kappa_psi}, \eqref{def:W_psi}, and \eqref{def:Q_psi} respectively. 
\end{rem*}
\noindent
This section is dedicated to the proofs of \cref{Theorem1,T1-MainTheorem}. 

\subsection{Proof of \cref{Theorem1} from \cref{T1-MainTheorem}} 
By comparing notation\footnotemark, one can verify that it suffices show that \eqref{Theorem1-xRange} implies \eqref{T1-xRange_L1} and \eqref{T1-xRange_L2-Quadratic} when $\psi$ is quadratic and  similarly implies  \eqref{T1-xRange_L1} and \eqref{T1-xRange_L2-Principal} when $\psi$ is principal. 
\footnotetext{Note that $\rho(\kp) = 1$ and $\Delta_{\psi} = b_{\psi} $ if $\psi$ is principal, and $\rho(\kp) = 2$ and $\Delta_{\psi} = L(1,\psi) b_{\psi}/2$ if $\psi$ is quadratic.}
 If $\psi$ is quadratic, then by \cref{kappa_LB_effective} and \cref{PR-RayClassGroup},
\begin{align*}
  (\kappa_{\psi}^{-1}+1)^4 W_{\psi} Q_{\psi}^4  h(\kq)^4 & \ll n_K^{12 n_K} d_K^{5+\tfrac{4}{n_K}} (\N\kq)^{5+\tfrac{2}{n_K}} \cdot h(\kq)^2 e^{2n_K}, \\
  (\kappa_{\psi}^{-1}+1)^{5} W_{\psi}^{5/2} Q_{\psi}^2  h(\kq)^2 & \ll n_K^{16 n_K} d_K^{6+\tfrac{5}{n_K}} (\N\kq)^{3.5+\tfrac{2.5}{n_K}} \cdot h(\kq)^2.
\end{align*}
One can therefore see by inspection that \eqref{Theorem1-xRange} indeed implies \eqref{T1-xRange_L1} and \eqref{T1-xRange_L2-Quadratic}.

If $\psi$ is principal, then similarly
\begin{align*}
  (\kappa_{\psi}^{-1}+1)^4 W_{\psi} Q_{\psi}^4  h(\kq)^4 & \ll n_K^{6 n_K} d_K^{3+\tfrac{4}{n_K}} (\N\kq)^3 \cdot h(\kq)^2 e^{2n_K}, \\
  (\kappa_{\psi}^{-1}+1)^3 W_{\psi}^{3/2} Q_{\psi}^2  h(\kq)^2 & \ll n_K^{5 n_K} d_K^{2+\tfrac{3}{n_K}} (\N\kq)^{0.5} \cdot h(\kq)^2.
\end{align*}
Again, one can see by inspection that \eqref{Theorem1-xRange} implies \eqref{T1-xRange_L1} and \eqref{T1-xRange_L2-Principal}.
\hfill \qed 
\subsection{Proof of \cref{T1-MainTheorem}} 
\label{subsec:T1-Proof} 
Let $y \in [1,10]$ be a parameter which is to be optimized later. Consider the sequence\footnotemark
\[
\cA^{(y)} =  \{ a_{\kn}^{(y)} \}_{\kn} \qquad \text{given by} \qquad a_{\kn}^{(y)} = \rho(\kn) e^{- y \N\kn/x} \cdot \mathbf{1}_{\{\kn \in \cC\}}
\]
where $\rho(\kn)$ is defined in \eqref{def:rho}. 
\footnotetext{Comparing with the notation of \eqref{AS-SieveSequence}, notice $\cA(x/y) = \cA^{(y)}$.}
For $B_{\delta} > 0$ sufficiently large, choose
\begin{equation*}
z = \{\kappa_{\psi}^{-1}+1\}^{1+\delta}  \cdot W_{\psi}^{1/2+\delta} (\N\kq)^{\delta} e^{B_{\delta} n_K}
\label{T1-SiftingLevel}
\end{equation*}
so $z$ indeed satisfies \eqref{AS-SieveParameters_z}. 
Analogous to \eqref{AS-SieveParameters}, define 
\begin{equation*}
D_y = \frac{(x/y)^{1-4\delta}}{h(\kq)^2 Q_{\psi}^{2+2\delta}}, \qquad \tau_y = \frac{\log D_y}{\log z}.
\label{T1-SieveParameters}
\end{equation*}
Furthermore, according to the notation of \cref{AS-LocalDensities}, denote
\[
X = b_{\psi} \kappa_{\psi} \frac{x}{h(\kq)}, \qquad V(z) = \prod_{\N\kp < z} (1-g(\kp)). 
\]
Now, by \cref{LambdaAndRho}, $\cA^{(y)}$ is supported on $\kn$ satisfying $\kp \mid \kn \implies \psi(\kp) =1$. Thus, we have the following Buchstab identity:
\begin{equation}
S(\cA^{(y)}, \sqrt{x}) = S(\cA^{(y)},z) - \sum_{\substack{ z \leq \N\kp < \sqrt{x} \\ \psi(\kp) = 1} } S(\cA^{(y)}_{\kp}, \kp).
\label{T1-BuchstabIdentity}
\end{equation}
Noting $a_{\kn}^{(y)} \leq a_{\kn}^{(1)}$, it follows $S(\cA^{(y)}_{\kp}, \kp) \leq S(\cA^{(1)}_{\kp}, \kp)$. Moreover,   $(1-\beta) \log x \ll_{\delta} \eta^{-1}$ by \eqref{T1-xRange_U}, and so from \eqref{T1-xRange_L1} and \cref{AS-LargePrimes,AS-PsiLargePrimes} it follows
\begin{equation}
\sum_{\substack{ z \leq \N\kp < \sqrt{x} \\ \psi(\kp) = 1} } S(\cA^{(y)}_{\kp}, \kp) \ll_{\delta} \eta^{-1} \cdot X V(z)
\label{T1-LargePrimes}
\end{equation}
provided $\eta \geq \eta(\delta)$. Assumption \eqref{T1-xRange_L1} allows us to apply \cref{AS-SmallPrimes} to $S(\cA^{(y)},z)$ so combined with  \eqref{T1-BuchstabIdentity} and \eqref{T1-LargePrimes}, we  deduce
\begin{equation}
\label{T1-Weighted}
S(\cA^{(y)}, \sqrt{x}) \geq \frac{1}{y}  \Big\{ 1- E_0(C_{\psi}; \tau_y)  + O(\delta) + O_{\delta}\big( \frac{1}{\log x}\big)  \Big\} \cdot XV(z)
\end{equation}
provided $\eta \geq \eta(\delta)$. It remains to convert the ``exponentially-weighted sieve" to the usual ``cutoff sieve". Observe
\begin{equation}
\label{T1-Cutoff}
\begin{aligned}
S(\cA^{(y)}, \sqrt{x}) & =  \sum_{\substack{ \kp \in \cC \\ \N\kp < x} } \rho(\kp) e^{-y\N\kp/x} + \sum_{\substack{ \kn \in \cC \\ (\kn, \kP(\sqrt{x}))=1 \\ \N\kn \geq x} } \rho(\kn) e^{- y \N\kn/x}, \\
& =  S_1 + S_2
\end{aligned}
\end{equation}
say. To complete the proof, it suffices to lower bound $S_1$ and so we require an upper bound on $S_2$. As $y \geq 1, z \leq \sqrt{x}$ and $x$ satisfies \eqref{T1-xRange_L1}, it follows by \cref{AS-SmallPrimes} that
\[
S_2 \leq e^{-y+1} S(\cA^{(1)}, z) \leq e^{-y+1} \Big\{1 + E_1(C_{\psi}; \tau_y) + O(\delta) + O_{\delta}\big( \frac{1}{\log x} \big) \Big\} \cdot X V(z).
\]
Using the above, \eqref{T1-Weighted}, and \eqref{T1-Cutoff}, we conclude for $\eta \geq \eta(\delta)$
\begin{equation}
S_1  \geq \frac{1}{C_{\psi}} \Big\{ \frac{1}{y}\Big(1- E_0(C_{\psi}; \tau_y) \Big) - e^{-y+1} \Big(1+ E_1(C_{\psi}; \tau_y) \Big) + O(\delta) + O_{\delta}\big( \frac{1}{\log x}\big) \Big\} \cdot X
\label{T1-LowerBound}
\end{equation}
after bounding $V(z)$ by \cref{AS-Dimension}. Finally, we consider cases.
\subsubsection{$\psi$ quadratic} Then \eqref{T1-xRange_L2-Quadratic} and our choice of $z$ imply $\tau_y > 5$, so $n_0(\tau_y) \geq 6$ and $n_1(\tau_y) \geq 5$. Hence, by the definitions in \cref{FundamentalLemma}, 
\begin{align*}
E_0(C_{\psi}; \tau_y) & \leq \big(\tfrac{1}{2}e^4 - \tfrac{11}{3}e^2 + \tfrac{1}{2}\big) \{ 1 + O(\delta) \} , \\
E_1(C_{\psi}; \tau_y) & \leq \big(\tfrac{1}{2}e^4 - \tfrac{10}{3}e^2 - \tfrac{1}{2}\big) \{ 1 + O(\delta) \}, 
\end{align*}
since $C_{\psi} = e^{2+\delta}$ by \cref{AS-Dimension}. Substituting these bounds into \eqref{T1-LowerBound}, choosing roughly optimally $y = 7.37$, and rescaling $\delta$ appropriately completes the proof of \cref{T1-MainTheorem} when $\psi$ is quadratic.
\subsubsection{$\psi$ principal} Then \eqref{T1-xRange_L2-Principal} and our choice of $z$ imply $\tau_y > 3$, so $n_0(\tau_y) \geq 4$ and $n_1(\tau_y) \geq 3$. Hence, by the definitions in \cref{FundamentalLemma}, 
\begin{align*}
E_0(C_{\psi}; \tau_y) & \leq \big(\tfrac{1}{2}e^2 - \tfrac{3}{2}e + \tfrac{1}{2}\big) \{ 1 + O(\delta) \} , \\
E_1(C_{\psi}; \tau_y) & \leq \big(\tfrac{1}{2}e^2 - e - \tfrac{1}{2}\big) \{ 1 + O(\delta) \} ,
\end{align*}
since $C_{\psi} = e^{1+\delta}$ by \cref{AS-Dimension}. Substituting these bounds into \eqref{T1-LowerBound}, choosing roughly optimally $y = 4.54$, and rescaling $\delta$ appropriately completes the proof of \cref{T1-MainTheorem} when $\psi$ is principal.

\bibliographystyle{alpha}
\bibliography{biblio}

\end{document}